\documentclass{amsart}
\usepackage{amsfonts}
\usepackage{amsmath,amssymb}
\usepackage{amsthm}
\usepackage{amscd}
\usepackage{graphics}
\usepackage{graphicx}

\theoremstyle{remark}{
\newtheorem{Def}{{\rm Definition}}
\newtheorem{Ex}{{\rm Example}}
\newtheorem{Rem}{{\rm Remark}}

}
\theoremstyle{plain}{

\newtheorem{Prop}{Proposition}
\newtheorem{Thm}{Theorem}
\newtheorem{MainThm}{Main Theorem}

}

\begin{document}
\title[On the images of special generic maps with simple structures]{The topologies and the differentiable structures of the images of special generic maps having simple structures}
\author{Naoki Kitazawa}
\address{Institute of Mathematics for Industry, Kyushu University, 744 Motooka, Nishi-ku Fukuoka 819-0395, Japan\\
 TEL (Office): +81-92-802-4402 \\
 FAX (Office): +81-92-802-4405 \\
}
\email{n-kitazawa@imi.kyushu-u.ac.jp}
\urladdr{https://naokikitazawa.github.io/NaokiKitazawa.html}
\maketitle
\begin{abstract}

{\it Special generic} maps are smooth maps at each singular point of which we can represent as $(x_1, \cdots, x_m) \mapsto (x_1,\cdots,x_{n-1},\sum_{k=n}^{m}{x_k}^2)$ for suitable coordinates.
Morse functions with exactly two singular points on homotopy spheres and canonical projections of unit spheres are special generic.
They are known to restrict the topologies and the differentiable structures of the manifolds in various situations. It also seems to be true that various manifolds admit such maps.

This article first presents a special generic map on a $7$-dimensional manifold and the image. 
This result also seems to present a new example of $7$-dimensional closed and simply-connected manifolds having non-vanishing triple Massey products and seems to be a new work related to similar works by Dranishnikov and Rudyak.
We also review results on vanishing of products of cohomology classes, previously obtained by the author. The images of special generic maps are smoothly immersed manifolds whose dimensions are equal to the dimensions of the manifolds of the targets. They know much of the topologies and the differentiable structures of the manifolds of the domains. The author studied the topologies of these images previously and studies on homology groups, cohomology rings and structures of them for special generic maps having simple structures are presented as new results.

\end{abstract}


\maketitle
\section{Introduction.}
\label{sec:1}
Throughout the present paper, manifolds and maps between manifolds are smooth or of class $C^{\infty}$. Diffeomorphisms on smooth manifolds are assumed to be smooth. We define the {\it diffeomorphism group} of a smooth manifold is the group of all diffeomorphisms there. We assume that the structure groups of bundles whose fibers are smooth manifolds are subgroups of the diffeomorphism groups unless otherwise stated. In other words the bundles are {\it smooth}.

${\mathbb{R}}^k$ denotes the $k$-dimensional Euclidean space for any integer $k \geq 1$ where
for ${\mathbb{R}}^1$, $\mathbb{R}$ is also used. It is regarded as a smooth manifold canonically and it is regarded as the Riemannian manifold endowed with the standard Euclidean metric. 
$||x|| \geq 0$ deno4tes the value of the standard Euclidean norm at $x \in {\mathbb{R}}^k$ and equivalently, the distance between $x$ and the origin $0$.  
The $k$-dimensional unit sphere $S^k$ is the set of all points $x$ in ${\mathbb{R}}^{k+1}$ satisfying $||x||=1$. It is regarded as a $k$-dimensional smooth closed submanifold with no boundary. A smooth manifold whose dimension is positive is said to be a {\it homotopy sphere} if it is homeomorphic to a unit sphere and a {\it standard sphere} if it is diffeomorphic to one. One-point sets are in considerable cases regarded as homotopy spheres and standard spheres. In the proof of Theorem \ref{thm:3}, one-point sets are regarded as submanifolds playing roles as homotopy spheres play. The $k$-dimensional unit disk $D^k$ is the set of all points $x$ in ${\mathbb{R}}^{k}$ satisfying $||x|| \leq 1$. It is regarded as a $k$-dimensional compact and smooth submanifold.

$\mathbb{N} \subset \mathbb{R}$ denotes the set of all positive integers.

The class of {\it linear} bundles is a subclass of the class of smooth bundles. A bundle is {\it linear} if the fiber is regarded as a unit sphere or a unit disk in a Euclidean space and the structure group acts linearly in a canonical way.

A {\it singular} point $p \in X$ of a smooth map $c:X \rightarrow Y$ is a point at which the rank of the differential ${dc}_p$ i4s smaller than both the dimensions $\dim X$ and $\dim Y$. $S(c)$ denotes the set of all singular points of $c$ (the {\it singular set} of $c$). We call $c(S(c))$ the {\it singular value set} of $c$. We call $Y-c(S(c))$ the {\it regular value set} of $c$. A {\it singular {\rm (}regular{\rm )} value} is a point in the singular (resp. regular) value set of the map.
\subsection{Special generic maps.}
{\it Special generic} maps are smooth maps at each singular point of which we can represent as $$(x_1, \cdots, x_m) \mapsto (x_1,\cdots,x_{n-1},\sum{k=n}^{m}{x_k}^2)$$
for suitable coordinates. Note that the restriction to the singular set is a smooth immersion of an ($n-1$)-dimensional closed smooth submanifold with no boundary, for example. Morse functions with exactly two singular points on homotopy spheres or functions playing important roles in so-called Reeb's theorem and canonical projections of unit spheres are special generic. The restrictions to the singular sets are embeddings in these cases.
As another example, an $m$-dimensional closed manifold $M$ represented as a connected sum of the $l>0$ manifolds diffeomorphic to each of the manifolds in $\{S^{l_j} \times S^{m-l_j}\}_{j=1}^{l}$ ($1 \leq l_j \leq n-1$) admits a special generic map $f:M \rightarrow {\mathbb{R}}^n$ such that $f {\mid}_{S(f)}$ is an embedding and that $f(M)$ is represented as a boundary connected sum of the $l$ manifolds diffeomorphic to each of the manifolds in $\{S^{l_j} \times D^{n-l_j}\}_{j=1}^{l}$ where (boundary) connected sums are considered in the smooth category. What makes special generic maps attractive is restrictions they pose on the topologies and the differentiable structures of the manifolds.

\begin{Ex}[\cite{calabi}, \cite{saeki}, \cite{saeki2}, \cite{saekisakuma} and \cite{wrazidlo} for example.]
An $m$-dimensional homotopy sphere always admits a special generic map into ${\mathbb{R}}^2$ for $m \neq 1,4$. If an $m$-dimensional homotopy sphere admits a special generic map into ${\mathbb{R}}^n$ for $n=m-3$, $n-2$ or $n-1$, then it is a standard sphere. Furthermore, $7$-dimensional oriented homotopy spheres of 14 types of all 28 types do not admit special generic maps into ${\mathbb{R}}^3$.   
Last, there exist pairs of mutually homeomorphic $4$-dimensional closed, connected and smooth manifolds such that exactly one of each pair admits a special generic map into ${\mathbb{R}}^3$.  
\end{Ex}

On the other hand, various manifolds other than the presented ones admit special generic maps into suitable Euclidean spaces. See \cite{kitazawa7}, \cite{kitazawa11} and \cite{kitazawa12} for example.

Moreover, the following theorem enables us to obtain a special generic map canonically from a smoothly immersed compact manifold of codimension zero.

\begin{Prop}[\cite{saeki}.]
\label{prop:1}
Let $m>n \geq 1$ be integers. 
\begin{enumerate}
\item
\label{prop:1.1}
Let $f:M \rightarrow {\mathbb{R}}^n$ be a special generic map on an $m$-dimensional closed and connected manifold $M$. Then there exists a 
smooth surjection $q_f:M \rightarrow W_f$ onto an $n$-dimensional compact manifold such that $q_f(S(f))= \partial W_f$ and a smooth immersion $\bar{f}:W_f \rightarrow {\mathbb{R}}^n$ such that $f=\bar{f} \circ q_f$.
\item
\label{prop:1.2}
Let ${\bar{f}}_n:W_n \rightarrow {\mathbb{R}}^n$ be a smooth immersion of an $n$-dimensional compact and connected manifold. Then there exists a special generic map $f:M \rightarrow {\mathbb{R}}^n$ on an $m$-dimensional closed, connected and orientable manifold $M$ such that the following properties hold where we abuse notation for the map $f$ just before.
\begin{enumerate}
\item $W_f$ is identified with $W_n$ and $\bar{f}={\bar{f}}_n$. 
\item There exists a small collar neighborhood $C(\partial W_f)$ such that the composition of $q_f {\mid}_{{q_f}^{-1}(C(\partial W_f))}$ with the canonical projection to $\partial W_f$ gives a trivial linear bundle whose fiber is diffeomorphic to $D^{m-n+1}$.
\item $q_f {\mid}_{{q_f}^{-1}(W_f-{\rm Int}(C(\partial W_f)))}$ gives a trivial smooth bundle whose fiber is diffeomorphic to $S^{m-n}$.
\end{enumerate}
\end{enumerate}
\end{Prop}

\subsection{Main theorems and the content of the present paper.}
In the present paper, first, we show the following result on a manifold satisfying a non-vanishing property for a product of cohomology classes. 

\begin{MainThm}
\label{mainthm:1}
There exist a $7$-dimensional closed and simply-connected manifold $M$ having a non-vanishing triple Massey product and a special generic map $f:M \rightarrow {\mathbb{R}}^6$. 
\end{MainThm}
In the next section, we prove this and present several remarks, without explicit exposition on {\it triple Massey products}. The third section is devoted to studies of compact manifolds $W_f$ for special generic maps $f$ for a newly introduced class of special generic maps. 
Main Theorem \ref{mainthm:1} motivates the author to introduce this class. Studies of this type such as \cite{kitazawa11} and \cite{kitazawa12} have been done by the author previously. Before them Nishioka also obtained a result \cite{nishioka} motivating the author to do related works. Furthermore, these images of the maps know much about the topologies and the differentiable structures of the closed manifolds of the domains. This general fact is presented in the last.
One of main theorems there is the following or Theorem \ref{thm:2}. We explain undefined notions and notation in the third section.

\begin{MainThm}
\label{mainthm:2}
Let $n \geq k \geq 2$ be integers.
\begin{enumerate}
\item \label{mthm:2.1}
For any $n$-dimensional compact, {\rm (}$k-1${\rm )}-connected and smooth manifold $X$ smoothly immersed {\rm (}embedded{\rm )} into ${\mathbb{R}}^n$ which is an {\rm SIE-$(K,{\mathcal{M}}_{\rm \mathcal{CS}Diff})$} {\rm (resp. SEE-$(K,{\mathcal{M}}_{\rm \mathcal{CS}Diff})$)}, any {\rm root} of the polyhedron $K$ $X$ collapses to consists of diffeomorphism types for {\rm(}$k-1${\rm )}-connected manifolds.  
\item \label{mthm:2.2}
In {\rm (\ref{mthm:2.1})}, suppose that $n \leq 3k$ and that the set ${\mathcal{M}}_{\rm \mathcal{CS}Diff}$ of diffeomorphism types is sufficiently large. Here we can choose a root $R_K \subset {\mathcal{M}}_{\rm \mathcal{CS}Diff}$ of $K$ and $K$ is obtained by a finite iteration of taking a bouquet starting from finitely many polyhedra satisfying either of the following two.
\begin{enumerate}
\item The product of a homotopy sphere in a diffeomorphism type in $R_K$ and a polyhedron obtained by a finite iteration of taking a bouquet starting from at least $2$ homotopy spheres in diffeomorphism types in $R_K$. In this case we can choose $R_K$ as a root consisting of diffeomorphism types for homotopy spheres.
\item A manifold represented as a connected sum of finitely many closed and {\rm (}$k-1${\rm )}-connected smooth manifolds in $R_K$ where the connected sum is taken in the smooth category.
\end{enumerate}  
\item \label{mthm:2.3}
Suppose that a set ${\mathcal{M}}_{\rm \mathcal{CS}Diff}$ of diffeomorphism types is sufficiently large. Let $K$ be an elementary polyhedron whose root can be a subset of ${\mathcal{M}}_{\rm \mathcal{CS}Diff}$. 
\begin{enumerate}
\item 
\label{mthm:2.3.1}
Assume also that $K$ is PL homeomorphic to one obtained by a finite iteration of taking a bouquet starting from finitely many polyhedra satisfying at least one of the following two conditions.
\begin{enumerate}
\item 
\label{mthm:2.3.1.1}
The product of a closed and {\rm (}$k-1${\rm )}-connected manifold $Y$ in diffeomorphism types in $R_K$ we can smoothly immerse {\rm (}resp. embed{\rm )} into ${\mathbb{R}}^{\dim Y+1}$ and a polyhedron obtained by a finite iteration of taking a bouquet starting from finitely many closed, {\rm(}$k-1${\rm )}-connected and smooth manifolds which can be smoothly embedded into one-dimensional higher Euclidean spaces where we choose a suitable root $R_K \subset {\mathcal{M}}_{\rm \mathcal{CS}Diff}$.
\item
\label{mthm:2.3.1.2}
A closed and {\rm (}$k-1${\rm )}-connected manifold in diffeomorphism types in $R_K$ we can smoothly immerse {\rm (}resp. embed{\rm )} into ${\mathbb{R}}^n$ where we choose a suitable root $R_K \subset {\mathcal{M}}_{\rm \mathcal{CS}Diff}$.
\end{enumerate}
\item
\label{mthm:2.3.2}
Assume also that $\dim K<n$ holds.
\end{enumerate}
Then there exists an $n$-dimensional compact, {\rm (}$k-1${\rm )}-connected and smooth manifold $X$ smoothly immersed {\rm (}embedded{\rm )} into ${\mathbb{R}}^n$ which is an SIE-$(K,{\mathcal{M}}_{\rm \mathcal{CS}Diff})$ {\rm (}resp. SEE-$(K,{\mathcal{M}}_{\rm \mathcal{CS}Diff})${\rm )}.
\end{enumerate}
\end{MainThm}
We also present new results Theorems \ref{thm:3}--\ref{thm:5}. Most of Theorem \ref{thm:3} are also essentially known results or corollaries to them or obtained by fundamental arguments on algebraic topology abd differential topology. Theorems \ref{thm:4} and \ref{thm:5} explicitize Theorem \ref{thm:2}. For them we apply explicit theory on  differential topology of manifolds whose dimensions are smaller than $6$. 

These studies can be also an important part of systematic studies on combinatorial and geometric aspects such as explicit classifications of compact manifolds with non-empty boundaries whose dimensions are general. 
For dimensions smaller than $5$, such manifolds are important tools, and actively studied as a part of an important topic of low dimensional geometry. For example, complementary spaces of (the interiors of small tubular neighborhoods of) links or 1-dimensional closed manifolds (smoothly embedded) in 3-dimensional closed, connected and orientable manifolds are fundamental objects in the knot theory and the theory of $3$-dimensional manifolds: see \cite{hempel} for example. As another example, $4$-dimensional manifolds obtained by attaching so-called {\it $1$-handles} and {\it $2$-handles} to the boundary of copies of the $4$-dimensional unit disk are fundamental objects in the theory of $4$-dimensional manifolds. The way of the attachment determines the differentiable structure of a closed and connected manifold obtained in a canonical way uniquely, for example. For the theory of $4$-dimensional manifolds, see \cite{gompgstipsicz} for example.
 
Recently, topological properties of complementary spaces of union of affine subspaces of real or complex vector spaces, or so-called ({\it affine}) {\it subspace arrangements}, are actively studied. For this, see \cite{zieglerzivaljevi} for example. See also \cite{ishikawaoyama} as a recent study on this topic for example. We can find various books and articles on this attractive topic. Such complementary spaces can be, in suitable senses, essentially regarded as compact manifolds with non-empty boundaries via suitable compactifications. As related studies, for example, see \cite{zariski}--\cite{zariski3} and see also \cite{guervilleballe} and other books or articles for studies on complex projective curves on complex projective spaces and the complementary spaces of the curves.  
However, we do not know explicit studies of combinatorial and geometric properties of these compact manifolds in general scenes and more general compact manifolds with non-empty boundaries so much.

\section{Proof of Main Theorem \ref{mainthm:1} and remarks.}
For {\rm (}{\it triple}{\rm )} {\it Massey products}, see \cite{massey} and see also \cite{kraines}, \cite{taylor} and \cite{taylor2}.
\begin{proof}[A proof of Main Theorem \ref{mainthm:1}.]
First we consider a so-called {\it Borromean link} of dimension $3$ in ${\mathbb{R}}^6$ in the smooth category and take a small closed tubular neighborhood. We consider a copy of the $6$-dimensional unit disk smoothly embedded into ${\mathbb{R}}^6$ and containing the closed tubular neighborhood in the interior. We remove the interior of the closed tubular neighborhood from the disc and obtain a new $6$-dimensional connected and compact manifold smoothly embedded into ${\mathbb{R}}^6$. For this resulting $6$-dimensional manifold, we apply Proposition \ref{prop:1} (\ref{prop:1.2}) and we have a desired special generic map on a $7$-dimensional closed and simply-connected manifold $M$ having a non-vanishing triple Massey product into ${\mathbb{R}}^6$.
\end{proof}
For {\it Borromean links}, see \cite{massey} for example.
\begin{Rem}
\label{rem:1}
According to \cite{kitazawa7}, no $7$-dimensional closed and simply-connected manifold $M$ having a non-vanishing triple Massey product admits a special generic map into $f:M \rightarrow {\mathbb{R}}^n$ for $n=1,2,3,4,5$.  
This result is presented as a corollary to a general theorem there. 
\end{Rem}

Related to Remark \ref{rem:1}, we present a main result of \cite{kitazawa9}, which is on vanishing of (cup) products of cohomology classes.
\begin{Thm}[\cite{kitazawa9}]
\label{thm:1}
Let $m>n \geq 1$ be integers and $l>0$ be an integer. Let $A$ be a commutative ring. Let $f:M \rightarrow N$ be a special generic map from an $m$-dimensional closed and connected manifold $M$ into an $n$-dimensional non-closed and connected manifold $N$ with no boundary. For any sequence $\{a_j\}_{j=1}^l \subset H^{\ast}(M;A)$ such that the degree of each element is smaller than or equal to $m-n$ and that the sum of the degrees is greater than or equal to $n$, the product ${\prod}_{j=1}^l a_j$ vanishes.   
\end{Thm}
\cite{kitazawa11} and \cite{kitazawa12} also present applications of this.
\begin{Rem}
\label{rem:2}
\cite{dranishnikovrudyak} and \cite{dranishnikovrudyak2} also present $7$-dimensional closed, simply-connected and smooth manifolds having non-vanishing triple Massey products. Their 3rd integral homology groups vanish. Their 2nd integral homology groups are free and the ranks are at least $4$. See also \cite{kitazawa7}. 
On the other hand, the 3rd integral homology group of $M$ in this new result or Main Theorem \ref{mainthm:1} vanishes. The 2nd integral homology group is free and the rank is $3$.
It is also known that the dimension of a closed, simply-connected and smooth manifold having non-vanishing triple Massey products must be greater than $6$: see these articles and see also \cite{fernandezMunoz} and \cite{miller} for example.
\end{Rem}
\section{A new class of special generic maps having simple structures and their images.}
If a closed and simply-connected manifold admits a special generic map $f$ whose codimension is negative into the line or the plane ${\mathbb{R}}^2$, then it is a homotopy sphere and $W_f$ in Proposition \ref{prop:1} is diffeomorphic to a closed interval or the $2$-dimensional unit disk. 
If a closed and simply-connected manifold admits a special generic map $f$ whose codimension is negative into ${\mathbb{R}}^3$, then they are diffeomorphic to a homotopy sphere or a manifold represented as a connected sum of total spaces of smooth bundles over $S^2$ whose fibers are homotopy spheres and $W_f$ is diffeomorphic to a manifold represented as a boundary connected sum of finitely many copies of $S^2 \times D^1$.
These studies are of \cite{saeki} and \cite{saeki2}. 
If a $5$-dimensional closed and simply-connected manifold admits a special generic map $f$ into ${\mathbb{R}}^4$, then it is a homotopy sphere or a manifold represented as a connected sum of total spaces of smooth bundles over $S^2$ whose fibers are diffeomorphic to $S^3$ where the connected sum is considered in the smooth category. On the other hand, such manifolds admit special generic maps into ${\mathbb{R}}^n$ for $n=3,4$. Furthermore, $W_f$ is a contractible $4$-dimensional manifold or a simply-connected manifold whose integral homology group and whose cohomology ring are isomorphic to those of a manifold represented as a boundary connected sum of finitely many copies of $S^2 \times D^2$ or $S^3 \times D^3$ where the boundary connected sum is considered in the smooth category. These studies on special generic maps on $5$-dimensional closed and simply-connected manifolds are by Nishioka \cite{nishioka}. We present a proposition of a smoothly immersed manifold of \cite{nishioka} again.

\begin{Prop}[\cite{nishioka}.]
\label{prop:2}
For a $4$-dimensional compact and simply-connected manifold $P$ smoothly immersed into ${\mathbb{R}}^4$, $H_2(P;\mathbb{Z})$ and $H_3(P;\mathbb{Z})$ are finitely generated and free.
\end{Prop}

Note also that \cite{barden} with \cite{smale} presents complete classifications of $5$-dimensional closed, simply-connected and topological manifolds in the topology, the PL, the piecewise smooth, and the smooth categories. Nishioka \cite{nishioka} applies Propositions \ref{prop:1} and \ref{prop:2}, investigate the homology groups of the $5$-dimensional closed and simply-connected manifolds of the domains of the resulting special generic maps and applies the classifications.

Our study of the present section is motivated by these facts on $W_f$ and Main Theorem \ref{mainthm:1}.
We can define an equivalence relation on the family of all smooth manifolds where two smooth manifolds are equivalent if and only if they are diffeomorphic or there exists a diffeomorphism between these two manifolds. For a smooth manifold $X$, we can uniquely define the class containing $X$. We call such a class a {\it diffeomorphism type} and the {\it diffeomorphism type} for $X$. We can define a similar equivalence relation on the family of all polyhedra where two polyhedra are equivalent if and only if they are PL homeomorphic or there exists a PL homeomorphism between these two polyhedra. For a polyhedron $X$, we can uniquely define the class containing $X$. We call such a class a {\it PL type} and the {\it PL type} for $X$. Unless otherwise stated, for a smooth manifold $X$, the {\it PL type} for $X$ is the PL type for the polyhedron $X$ compatible with $X$ as a smooth manifold. Such a polyhedron is known to exist and it is unique if we consider the PL type for the polyhedron. A {\it PL sphere} is a manifold regarded as a polyhedron compatible with a homotopy sphere which is not a $4$-dimensional homotopy sphere not being a standard sphere. The PL types for PL spheres of a fixed dimension is unique. This does not depend on the differentiable structures of the homotopy spheres.

For the PL category and the piecewise smooth category, which are known to be equivalent, there remain more to present. However we do not need to know them so much. We introduce some of them.

For example, PL maps, which are morphisms in the PL category, are regarded as piecewise smooth maps, which are morphisms in the piecewise smooth category. Smooth maps are regarded as piecewise smooth maps and PL maps when we regard the smooth manifolds as suitable polyhedra by considering the PL types. It is also known as a fundamental principle that a continuous map between polyhedra is approximated by a PL maps by a suitable homotopy.

${\mathcal{PL}}$ denotes the set of all PL types for all polyhedra. Let $\mathcal{M}_{\rm \mathcal{CS}Diff}$ be a set of several diffeomorphism types of  closed, connected and smooth manifolds. We consider a non-negative integer $l \in \{0\} \sqcup \mathbb{N}$ and a sequence $s_{\mathcal{M}_{\rm \mathcal{CS}Diff}}:\{0\} \sqcup {\mathbb{N}}_{\leq l} \rightarrow \mathcal{M}_{\rm \mathcal{CS}Diff} \times \{1\} \subset (\mathcal{M}_{\rm \mathcal{CS}Diff} \times \{1\}) \sqcup (\mathcal{PL} \times \{0\})$ where ${\mathbb{N}}_{\leq l_0}:=\{a \in \mathbb{N} \mid a \leq l_0\}$ for $l_0 \in \mathbb{R}$.  
We consider an iteration of the following three steps starting from $s_k:=s_{\mathcal{M}_{\rm \mathcal{CS}Diff}}$ with $k=0$.

\begin{enumerate}
\item
\label{s;1}
Take $s_k$ ($k \in \{0\} \sqcup {\mathbb{N}}_{\leq l}$). If $k=l$, then set $s_{l,0}:=s_l(0)$ and finish the present steps.
\item
\label{s:2}
Choose two distinct numbers $k_1,k_2 \in \{0\} \sqcup {\mathbb{N}}_{\leq l-k}$ satisfying $k_1<k_2$. Define $s_{k+1}:\{0\} \sqcup {\mathbb{N}}_{\leq l-k-1} \rightarrow (\mathcal{M}_{\rm \mathcal{CS}Diff} \times \{1\}) \sqcup (\mathcal{PL} \times \{0\})$ so that the following conditions hold on the set $\{0\} \sqcup ({\mathbb{N}}_{\leq l-k-1}-\{l-k-1\})$ and for $j=l-k-1 \in \{0\} \sqcup {\mathbb{N}}_{\leq l-k-1}$.
\begin{enumerate} 
\item
\label{s:2a}
For $0 \leq j < k_1-1$, $s_{k+1}(j):=s_k(j)$,
\item
\label{s:2b}
For $k_1-1 \leq j< k_2-1$, $s_{k+1}(j):=s_k(j+1)$.
\item
\label{s:2c}
For $k_2-1 \leq j <l-k-1$, $s_{k+1}(j):=s_k(j+2)$.
\item
\label{s:2d}
For $j=l-k-1$, either of the following three holds where ${\rm pr}_{\{0,1\}}$ denotes the projection to the second component.
\begin{enumerate}
\item
\label{s:2d1}
$s_{k+1}(j)$ is the pair of the PL type for a bouquet of a two polyhedra and $0$: the polyhedra are a polyhedron in 
the first component of $s_k(k_1)$ and one in the first component of $s_k(k_2)$, respectively. 
\item
\label{s:2d2}
$s_{k+1}(j)$ is the pair of the PL type for a product of a smooth manifold and a polyhedron and $0$ and $({\rm pr}_{\{0,1\}}(s_k(k_1)),{\rm pr}_{\{0,1\}}(s_k(k_2))) \neq (0,0)$: the polyhedra are a polyhedron or a smooth manifold in the first component of $s_k(k_1)$ and one in the first component of $s_k(k_2)$, respectively. 
\item
\label{s:2d3}
$s_{k+1}(j)$ is the pair of the diffeomorphism type for a smooth manifold represented as a connected sum of a smooth manifold and another smooth manifold and $1$ and $({\rm pr}_{\{0,1\}}(s_k(k_1)),{\rm pr}_{\{0,1\}}(s_k(k_2))) = (1,1)$: the smooth manifolds are a smooth manifold in the first component of $s_k(k_1)$ and one in the first component of $s_k(k_2)$, respectively. Moreover, the connected sum is considered in the smooth category.
\end{enumerate}
For each of the two manifolds or polyhedra to obtain a new manifold or a polyhedron, we can define a canonical embedding into the new space in the PL category for the first two cases. We call these embeddings {\it trace embeddings} and if the value ${\rm pr}_{\{0,1\}}(s_k(k_j))$ ($j=1,2$) is $1$, then we call this embedding a {\it special} trace embedding.
\end{enumerate}
\item
\label{s:3}
Return to the first step here by taking $s_{k+1}$ instead. 
\end{enumerate}
We consider a polyhedron the PL type for which is the first component of the value $s_{l,0}$. We call this polyhedron or a polyhedron obtained in this way an {\it elementary polyhedron generated by $\mathcal{M}_{\rm \mathcal{CS}Diff}$}. We call this sequence $s_{\mathcal{M}_{\rm \mathcal{CS}Diff}}$ used to obtain the polyhedron a {\it root} of the polyhedron. We also call the image of the composition of a root of the polyhedron with the projection to the first component a {\it root}. We say that the sequence of the pairs of trace embeddings defined in each step of the first two types in (\ref{s:2d}) is {\it associated with} the polyhedron: the length is same as the time of steps of the first two types in (\ref{s:2d}).
\begin{Def}
\label{def:1}
Let $n \geq 1$ be an integer. An $n$-dimensional compact, connected and smooth manifold $X$ smoothly immersed into ${\mathbb{R}}^n$ is said to be a {\it smoothly immersed elementary manifold of $(K,{\mathcal{M}}_{\rm \mathcal{CS}Diff})$} or {\it SIE-$(K,{\mathcal{M}}_{\rm \mathcal{CS}Diff})$} if the following properties hold.
\begin{enumerate}
\item $X$ collapses to an elementary polyhedron $K$ generated by ${\mathcal{M}}_{\rm \mathcal{CS}Diff}$ in the PL category.
\item Any special embedding in any pair of a suitable sequence of trace embeddings associated with $K$ is smooth as an embedding into $X$.
\item If for $K$ and a suitable procedure to obtain this before, the second component of the value $s_{l,0}$ is $1$, then we also assume the previous embedding of the smooth manifold $K$ into $X$ to be a smooth embedding.
\end{enumerate}
If ''immersed'' is replaced by ''embedded'', then it is said to be a {\it smoothly embedded elementary manifold of $(K,{\mathcal{M}}_{\rm \mathcal{CS}Diff})$} or {\it SEE-$(K,{\mathcal{M}}_{\rm \mathcal{CS}Diff})$}.
Furthermore, in Proposition \ref{prop:1}, if $\bar{f}$ is an immersion and $W_f$ is an SIE-$(K,{\mathcal{M}}_{\rm \mathcal{CS}Diff})$ (an embedding and $W_f$ is an SEE-$(K,{\mathcal{M}}_{\rm \mathcal{CS}Diff})$), then the special generic map $f$ is said to be an {\it SIE-$(K,{\mathcal{M}}_{\rm \mathcal{CS}Diff})$} (resp. {\it SEE-$(K,{\mathcal{M}}_{\rm \mathcal{CS}Diff})$}).
\end{Def}
In Theorem 6 of \cite{kitazawa9}, we can construct explicit special generic maps satisfying definitions in Definition \ref{def:1} for example. The $n$-dimensional compact manifolds are regarded as so-called {\it regular neighborhoods} of $K$ in the smooth category. \cite{hirsch} gives a rigorous definition of this for example.

The following is Main Theorem \ref{mainthm:2}.
\begin{Thm}
\label{thm:2}
Let $n \geq k \geq 2$ be integers.
\begin{enumerate}
\item \label{thm:2.1}
For any $n$-dimensional compact, {\rm (}$k-1${\rm )}-connected and smooth manifold $X$ smoothly immersed {\rm (}embedded{\rm )} into ${\mathbb{R}}^n$ which is an SIE-$(K,{\mathcal{M}}_{\rm \mathcal{CS}Diff})$ {\rm (}resp. SEE-$(K,{\mathcal{M}}_{\rm \mathcal{CS}Diff})${\rm )}, any {\rm root} of the polyhedron $K$ $X$ collapses to consists of diffeomorphism types for {\rm(}$k-1${\rm )}-connected manifolds.  
\item \label{thm:2.2}
In {\rm (\ref{thm:2.1})}, suppose that $n \leq 3k$ and that the set ${\mathcal{M}}_{\rm \mathcal{CS}Diff}$ of diffeomorphism types is sufficiently large. Then we can choose a root $R_K \subset {\mathcal{M}}_{\rm \mathcal{CS}Diff}$ of $K$ and $K$ is obtained by a finite iteration of taking a bouquet starting from finitely many polyhedra satisfying either of the following two.
\begin{enumerate}
\item The product of a homotopy sphere in a diffeomorphism type in $R_K$ and a polyhedron obtained by a finite iteration of taking a bouquet starting from at least $2$ homotopy spheres in diffeomorphism types in $R_K$. In this case we can choose $R_K$ as a root consisting of diffeomorphism types for homotopy spheres.
\item A manifold represented as a connected sum of finitely many closed and {\rm (}$k-1${\rm )}-connected smooth manifolds in $R_K$ where the connected sum is taken in the smooth category.
\end{enumerate}  
\item \label{thm:2.3}
Suppose that a set ${\mathcal{M}}_{\rm \mathcal{CS}Diff}$ of diffeomorphism types is sufficiently large. Let $K$ be an elementary polyhedron whose root can be a subset of ${\mathcal{M}}_{\rm \mathcal{CS}Diff}$. 
\begin{enumerate}
\item
\label{thm:2.3.1}
Assume also that $K$ is PL homeomorphic to one obtained by a finite iteration of taking a bouquet starting from finitely many polyhedra satisfying at least one of the following two conditions.
\begin{enumerate}
\item
\label{thm:2.3.1.1}
The product of a closed and {\rm (}$k-1${\rm )}-connected manifold $Y$ in diffeomorphism types in $R_K$ we can smoothly immerse {\rm (}resp. embed{\rm )} into ${\mathbb{R}}^{\dim Y+1}$ and a polyhedron obtained by a finite iteration of taking a bouquet starting from finitely many closed, {\rm(}$k-1${\rm )}-connected and smooth manifolds which can be smoothly embedded into one-dimensional higher Euclidean spaces where we choose a suitable root $R_K \subset {\mathcal{M}}_{\rm \mathcal{CS}Diff}$.
\item
\label{thm:2.3.1.2}
A closed and {\rm (}$k-1${\rm )}-connected manifold in diffeomorphism types in $R_K$ we can smoothly immerse {\rm (}resp. embed{\rm )} into ${\mathbb{R}}^n$ where we choose a suitable root $R_K \subset {\mathcal{M}}_{\rm \mathcal{CS}Diff}$.
\end{enumerate}
\item
\label{thm:2.3.2}
Assume also that $\dim K<n$ holds.
\end{enumerate}
Then there exists an $n$-dimensional compact, {\rm (}$k-1${\rm )}-connected and smooth manifold $X$ smoothly immersed {\rm (}embedded{\rm )} into ${\mathbb{R}}^n$ which is an SIE-$(K,{\mathcal{M}}_{\rm \mathcal{CS}Diff})$ {\rm (}resp. SEE-$(K,{\mathcal{M}}_{\rm \mathcal{CS}Diff})${\rm )}.
\end{enumerate}
\end{Thm}

\begin{proof}
We prove (\ref{thm:2.1}). By the definition, $K$ is obtained by a finite iteration of taking a bouquet, a product, or a connected sum starting from smooth manifolds in diffeomorphisms in the root.
By virtue of the definitions of notions on these polyhedra and fundamental theorems on algebraic topology for example, we immediately have (\ref{thm:2.1}). 

We prove (\ref{thm:2.2}).
Let $K$ be a polyhedron which is not PL homeomorphic to any bouquet of two polyhedra which are not one-point sets there. In other words, in this case we never use an operation of taking a bouquet in defining $s_l$ and $s_{l,0}$ in the presented procedures for obtaining $K$.

Note also that $\dim K <n \leq 3k$ holds by the assumption. By the definitions of notions on polyhedra related to this scene, for $K$ we can argue as either of the following two cases.\\
 \\
Case A. \\
$K$ is represented as a connected sum of closed and ($k-1$)-connected manifolds where the connected sum is taken in the smooth category. We never use an operation of taking a product in defining $s_l$ and $s_{l,0}$ in the presented procedures for obtaining $K$. By the assumption on the piecewise, this must be a homotopy sphere of dimension at least $k$ and at most $2k-1$ or a closed and ($k-1$)-connected manifold dimension at least $2k$ and at most $3k-1$. In the latter case, this must be either of the following two. Furthermore, $R_K \subset {\mathcal{M}}_{\rm \mathcal{CS}Diff}$ is chosen suitably.
\begin{itemize}
\item This is a homotopy sphere and represented as a connected sum of homotopy spheres in diffeomorphism types in $R_K$ where the connected sum is taken in the smooth category.
\item This is represented as a connected sum of finitely many closed and {\rm (}$k-1${\rm )}-connected smooth manifolds in $R_K$ containing at least one manifold which is not a homotopy sphere. Of course the connected sum is taken in the smooth category.
\end{itemize}
Case B. \\
$K$ is a product of the following two spaces each of which is a smooth manifold or merely a polyhedron where we choose a root $R_K \subset {\mathcal{M}}_{\rm \mathcal{CS}Diff}$ suitably. Moreover, their dimensions are at least $k$ by the assumption and we use an operation of taking a product in defining $s_l$ and $s_{l,0}$ in the presented procedures for obtaining $K$.
\begin{itemize}
\item One is a closed and smooth manifold whose dimension is at most $2k-1$. This must be ($k-1$)-connected by the assumption on the connectivity and as a result a homotopy sphere in $R_K$. 
\item The other is a polyhedron, whose dimension is at least $k$ and at most $2k-1$. By the conditions on the dimensions and the connectivity, this must be a bouquet of homotopy spheres in $R_K$. The number of homotopy spheres used in the bouquet is greater than or equal to $2$ in general.
\end{itemize}
This yields the fact (\ref{thm:2.2}). 

We prove (\ref{thm:2.3}). 
We explain about a polyhedron of (\ref{thm:2.3.1.1}), the product of the former manifold and a copy of the ($n-\dim Y$)-dimensional unit disk can be, by the definition, smoothly immersed or embedded into ${\mathbb{R}}^n$. The latter polyhedron can be embedded into the interior of a copy of $D^{n-\dim Y}$ as a subpolyhedron. As a result, we can take a desired smooth manifold $X$ as a suitable regular neighborhood of the product of a suitable two polyhedra regarded as a subpolyhedron in ${\mathbb{R}}^n$.
For a closed manifold of (\ref{thm:2.3.1.2}) represented by the connected sum, we can take a suitable smooth manifold $X$ by the definition.
For a general case, we can take a boundary connected sum of such manifolds and we have a desired manifold $X$ where the boundary connected sum is taken in the smooth category.

\end{proof}

\begin{Def}
\label{def:2}
An $n$-dimensional compact, connected and smooth manifold $X$ smoothly immersed into ${\mathbb{R}}^n$
is said to be {\it essentially an SIE-$(K,{\mathcal{M}}_{\rm \mathcal{CS}Diff})$} if it is obtained by restricting an {\it SIE-$(K,{\mathcal{M}}_{\rm \mathcal{CS}Diff})$} $X_0$ to the complementary set of
the interior of the disjoint union of finitely many disjoint small regular neighborhoods of elementary polyhedra in ${\rm Int}\ X_0$ generated by $\mathcal{M}_{\rm \mathcal{CS}Diff}$ where regular neighborhoods are considered in the smooth category. 
Furtheremore, we define the following notions.
\begin{enumerate}
\item We call the elementary polyhedra in ${\rm Int}\ X_0$ {\it holes} for the original SIE-$(K,{\mathcal{M}}_{\rm \mathcal{CS}Diff})$.
\item We call the original immersed manifold $X_0$ the {\it host} for the immersed manifold which is essentially an SIE-$(K,{\mathcal{M}}_{\rm \mathcal{CS}Diff})$.
\item As Definition \ref{def:1}, we define similar notions by adding ''{\it essentially}'' in the beginning.
Moreover, if holes are always bouquets of homotopy spheres or one-point sets, then we add ''{\it very essentially}'' instead of ''essentially''.
\end{enumerate}
\end{Def}
We present another theorem as Theorem \ref{thm:3}. Some of the theorem are very fundamental. Some of them are or may be essentially equivalent to known results or corollaries to them in \cite{kitazawa0.1}--\cite{kitazawa2}, \cite{kitazawa}, \cite{kitazawa2}--\cite{kitazawa7}. Remark \ref{rem:3} also gives related expositions. 
For example, (\ref{thm:3.3}) may be also shown via theory in these studies and related studies. More precisely, \cite{kitazawa1.1}--\cite{kitazawa1.3} and \cite{kitazawa10} may also help. The author expects that some methods of construction of special generic maps into higher dimensional Euclidean spaces from given fold maps such that preimages of regular values are disjoint unions of standard spheres and that are represented as the compositions of the resulting special generic maps with canonical projections.
For a real number $r$, $[r]$ denotes the greatest integer satisfying $[r] \leq r$. A {\it PID} means a so-called principal ideal domain having a unique identity element different from the zero element. 
For a polyhedron, if a homology class is equal to the class realized as the value of the homomorphism induced by a PL embedding of a closed, connected and orientable PL manifold at a {\it fundamental class}, then we say that the class is {\it represented by the manifold}. A {\it fundamental class} of a closed, connected and orientable PL manifold is a homology class generating the homology group of the top homology group of the manifold and compatible with a suitable orientation. 
\begin{Thm}
\label{thm:3}
Let $n \geq k \geq 2$ be integers. Let $K$ be a one-point set. Let $A$ be a PID. Let $\{G_j\}_{j=1}^{n}$ be a sequence of finitely generated free modules of length $n$ over $A$ such that $G_j$ is trivial for $1 \leq j \leq k-1$ and $j=n$ and that the inequality ${\rm rank}\ G_{n-1}>0$ holds.
Assume also that the set ${\mathcal{M}}_{\rm \mathcal{CS}Diff}$ is sufficiently large. We have the following five.

\begin{enumerate}
\item
\label{thm:3.1}
 Let an $n$-dimensional compact, connected and smooth manifold $X$ smoothly immersed into ${\mathbb{R}}^n$ be very essentially an SIE-$(K,{\mathcal{M}}_{\rm \mathcal{CS}Diff})$ and {\rm (}$k-1${\rm )}-connected. Then the holes are bouquets of PL spheres whose dimensions are at most $n-k-1$ or one-point sets and the host is diffeomorphic to the unit disk of dimension $n$.  
\item
\label{thm:3.2}
 For a suitable disjoint union of bouquets of PL spheres, which is denoted by $K_0$, we can obtain an $n$-dimensional compact, connected and smooth manifold $X$ smoothly embedded into ${\mathbb{R}}^n$ which is very essentially an SEE-$(K,{\mathcal{M}}_{\rm \mathcal{CS}Diff})$ and {\rm (}$k-1${\rm )}-connected satisfying the following properties. 
\begin{enumerate}
\item The disjoint union of all holes are PL homeomorphic to $K_0$.
\item $H_j(X;A)$ and $H^j(X;A)$ are isomorphic to $G_j$
\end{enumerate}
\item
\label{thm:3.3}
Suppose ${\Sigma}_{j=1}^{n-2} {\rm rank}\ G_j \leq {\rm rank}\ G_{n-1}$.
Let $\{a_{j}\}_{j=1}^{l}$ be a sequence of integers of length $l:={\Sigma}_{j=1}^{[\frac{n-1}{2}]} ({\rm rank}\ G_{j}\ {\rm rank}\ G_{n-1-j})>0$.

For a suitable disjoint union of PL spheres or one-point sets, which is denoted by $K_0$, we can obtain an $n$-dimensional compact, connected and smooth manifold $X$ smoothly embedded into ${\mathbb{R}}^n$ which is very essentially an SEE-$(K,{\mathcal{M}}_{\rm \mathcal{CS}Diff})$ and {\rm (}$k-1${\rm )}-connected satisfying the following properties. 
\begin{enumerate}
\item
\label{thm:3.3.1}
The disjoint union of all holes is PL homeomorphic to $K_0$.
\item
\label{thm:3.3.2}
$H_j(X;A)$ and $H^j(X;A)$ are isomorphic to $G_j$.
\item
\label{thm:3.3.3}
The following two hold for a suitable basis $\{e_{j_1,j_2}\}_{j_2=1}^{{\rm rank}\ G_{j_1}}$ of $H^{j_1}(X;A)$ for $1 \leq j_1 \leq n-1$.
\begin{enumerate}
\item
\label{thm:3.3.3.1}
The cup product of $e_{a_1,a_2}$ and $e_{b_1,b_2}$ is $$a_{{\Sigma}_{j=1}^{a_1-1} {\rm rank}\ G_{j}\ {\rm rank}\ G_{n-1-j}+{\rm rank}\ (G_{a_1})(b_2-1)+a_2}(e_{n-1,{\Sigma}_{j=1}^{a_1-1} {\rm rank}\ (G_{j})+a_2}+e_{n-1,{\Sigma}_{j=1}^{b_1-1} {\rm rank}\ (G_{j})+b_2}) \in H^{n-1}(X;A)$$ in the case $a_1<b_1$ and $a_1+b_1=n-1$ and zero in the case $a_1>0$, $b_1>0$ and $a_1+b_1 \neq n-1$. 
\item
\label{thm:3.3.3.2}
Suppose that $n$ is odd. Suppose that the sequence 
$\{a_{0,j}\}_{j=1}^{\frac{({\rm rank}\ G_{(\frac{n-1}{2})}-1){\rm rank}\ G_{(\frac{n-1}{2})}}{2}}$ of integers of length $l_0:=\frac{({\rm rank}\ G_{(\frac{n-1}{2})}-1){\rm rank}\ G_{(\frac{n-1}{2})}}{2}$ is given. The cup product of $e_{\frac{n-1}{2},a}$ and $e_{\frac{n-1}{2},b}$ is $$a_{0,{\Sigma}_{j=1}^{a-1} (l_0-j)+(b-a)}(e_{n-1,{\Sigma}_{j=1}^{\frac{n-1}{2}-1} {\rm rank}\ (G_{j})+a}+e_{n-1,{\Sigma}_{j=1}^{\frac{n-1}{2}-1} {\rm rank}\ (G_{j})+b}) \in H^{n-1}(X;A)$$ for $a<b$ and zero for $a=b$. 
\end{enumerate}
\end{enumerate}
\item
\label{thm:3.4}
In the previous statement, conversely, for any $n$-dimensional compact, connected and smooth manifold $X$ smoothly immersed into ${\mathbb{R}}^n$ which is very essentially an SIE-$(K,{\mathcal{M}}_{\rm \mathcal{CS}Diff})$ and {\rm (}$k-1${\rm )}-connected and whose holes are always homotopy spheres, the cohomology ring whose coefficient ring is $A$ is isomorphic to one obtained in this presented way.
\item
\label{thm:3.5}
Let an $n$-dimensional compact, connected and smooth manifold $X$ smoothly immersed into ${\mathbb{R}}^n$ be very essentially an SIE-$(K,{\mathcal{M}}_{\rm \mathcal{CS}Diff})$ and {\rm (}$k-1${\rm )}-connected. If $n<3k$, then for
any $n$-dimensional compact, connected and smooth manifold $X$ smoothly immersed into ${\mathbb{R}}^n$ which is very essentially an SIE-$(K,{\mathcal{M}}_{\rm \mathcal{CS}Diff})$ and {\rm (}$k-1${\rm )}-connected, any triple Massey product vanishes.
\end{enumerate}
\end{Thm}

\begin{proof}
We prove (\ref{thm:3.1}).
Let $D_0$ denote the host. This collapses to a one-point set $K$ by the assumption and as a result, diffeomorphic to the $n$-dimensional unit disk. 
Let $K_0$ be the disjoint union of all holes and $N(K_0)$ denote a small regular neighborhood. We have the following exact sequence for the PID $A$: 
$$\begin{CD}
@>   >> H_i(\partial N(K_0);A) @> >> H_i(N(K_0);A) \oplus H_i(D_0-{\rm Int} N(K_0);A)
\end{CD}$$

$$\begin{CD}
@>   >>  H_i(D_0;A) @>    >> H_{i-1}(\partial N(K_0);A) 
\end{CD}$$

$$\begin{CD}
@>   >> H_{i-1}(N(K_0);A) \oplus H_{i-1}(D_0-{\rm Int} N(K_0);A) @>   >>
\end{CD}.$$

$D_0$ is diffeomorphic to a unit disk. We can also regard $N(K_0)$ as a manifold represented as a boundary connected sum of finitely many copies of manifolds diffeomorphic to $S^{n-j} \times D^{j}$ for $n-j>0$ or $D^j$ for $n=j$ where the connected sum is taken in the smooth category. 
The sequence yields $j \geq k+1$. As a result holes are regarded as bouquets of PL spheres whose dimensions are at most $n-k-1$ or one-point sets.

We prove (\ref{thm:3.2}). First we consider the case ${\Sigma}_{j=1}^{n-2}\ {\rm rank}\ G_j \leq {\rm rank}\ G_{n-1}$. We consider a suitable disjoint union of PL spheres consisting of exactly ${\rm rank}\ G_j$ ($n-j-1$)-dimensional PL spheres for $1 \leq j < n-1$ and ${\rm rank}\ G_{n-1}-{\Sigma}_{j=1}^{n-2}\ {\rm rank}\ G_j$ one-point sets.
We can take the holes diffeomorphic to them. This completes the proof for the case ${\Sigma}_{j=1}^{n-2}\ {\rm rank}\ G_j \leq {\rm rank}\ G_{n-1}$. For a general case, we choose $l_{n-1}>0$ of these spheres or one-point sets and consider a bouquet instead of the family of $l_{n-1}$ manifolds consisting of spheres or one-point sets. This completes the proof for any case ${\rm rank}\ G_{n-1}>0$. For each hole, take a small regular neighborhood and its boundary. The homology classes represented  by these boundaries form a basis of $H_{n-1} (D_0-{\rm Int} N(K_0);A)$. 

We prove (\ref{thm:3.3}). We prove this based on fundamental methods of calculating homology groups and cohomology rings.
We consider a suitable disjoint union of PL spheres consisting of exactly ${\rm rank}\ G_j$ ($n-j-1$)-dimensional PL sphere for $1 \leq j<n-1$ and ${\rm rank}\ G_{n-1}-{\Sigma}_{j=1}^{n-2}\ {\rm rank}\ G_j$ one-point sets.  
We can embed this smoothly in the interior of a manifold $D_0$ before, diffeomorphic to the $n$-dimensional unit disk and smoothly embedded in ${\mathbb{R}}^n$.

We first smoothly embed the disjoint union of the PL spheres whose dimensions are greater than $[\frac{n-1}{2}]$. For such an arbitrary embedding we find a suitable smooth embedding of the disjoint union of the PL spheres whose dimensions are smaller than $[\frac{n-1}{2}]$. More precisely, we can embed the disjoint union smoothly so that 
for the $a_2$-th ($n-a_1-1$)-dimensional PL sphere $K_1$ there and the $b_2$-th ($n-b_1-1$)-dimensional PL sphere $K_2$ there, the following properties hold under the conditions $a_1<b_1$ and $a_1+b_1=n-1$.  
Facts on homology groups are essentially (\ref{thm:3.2}) and here we concentrate on cohomology groups and rings.

\begin{itemize}
\item For a small regular neighborhood $N(K_1)$ of $K_1$, it is regarded as a closed tubular neighborhood and a trivial linear bundle whose fiber is diffeomorphic to the ($a_1+1$)-dimensional unit disk.
\item For a small regular neighborhood $N(K_2)$ of $K_2$, it is regarded as a closed tubular neighborhood and a trivial linear bundle whose fiber is diffeomorphic to the ($b_1+1$)-dimensional unit disk.
\item For a fiber of the bundle $\partial N(K_1)$ obtained as a subbundle of the bundle before, we can take a homology class $e_{{\rm h},a_1,a_2}$ of degree $a_1=n-1-b_1$ represented by the fiber. 
For a fiber of the bundle $\partial N(K_2)$ obtained as a subbundle of the bundle before, we can take a homology class $e_{{\rm h},b_1,b_2}$ of degree $b_1=n-1-a_1$ represented by the fiber. 

\end{itemize}

Moreover, for each fixed pair $(b_1,b_2)$ and each $K_2$, we can construct the embedding so that $${\Sigma}_{a_2=1}^{{\rm rank}\ G_{a_1}} a_{{\Sigma}_{j=1}^{a_1-1}\ ({\rm rank}\ G_{j} {\rm rank}\ G_{n-1-j})+{\rm rank}\ (G_{a_1})(b_2-1)+a_2}{e_{{\rm h},a_1,a_2}}$$ is represented by $K_2$.

We can take a desired cohomology class $e_{a_1,a_2}$ and desired bases satisfying the following properties by the topological properties of the manifolds. $1 \in A$ denotes the unique identity element and it is different from the zero element $0 \in A$.
\begin{itemize}
\item $e_{a_1,a_2}(e_{{\rm h},a_1,a_2})=1$.
\item $e_{a_1,a_2}(e_{{\rm h},a_1,\tilde{a_{2}}})=0$ for $\tilde{a_{2}} \neq a_2$.
\end{itemize}
We can take a desired cohomology class $e_{b_1,b_2}$ and desired bases similarly.
For $a_j$ and $b_j$ with $j=1,2$ as before, we can take a homology class $$e_{{\rm h},n-1,{\Sigma}_{j=1}^{a_1-1} {\rm rank}\ G_{j}+a_2}$$ represented by the boundary $\partial N(K_1)$ and a homology class
$$e_{{\rm h},n-1,{\Sigma}_{j=1}^{b_1-1} {\rm rank}\ G_{j}+b_2}$$ represented by the boundary $\partial N(K_2)$.
We can take a desired basis $\{e_{n-1,j_2}\}$ of $H^{n-1}(X;A)$ in a way simiar to the way before. For a suitable one we have that the cup product of $e_{a_1,a_2}$ and $e_{b_1,b_2}$ is

$$a_{{\Sigma}_{j=1}^{a_1-1} {\rm rank}\ G_{j}\ {\rm rank}\ G_{n-1-j}+{\rm rank}\ (G_{a_1})(b_2-1)+a_2}(e_{n-1,{\Sigma}_{j=1}^{a_1-1} {\rm rank}\ (G_{j})+a_2}+e_{n-1,{\Sigma}_{j=1}^{b_1-1} {\rm rank}\ (G_{j})+b_2}) \in H^{n-1}(X;A)$$
by the topological properties of the manifolds.

In the case $n$ is odd, in addition we can also embed ($\frac{n-1}{2}$)-dimensional PL spheres satisfying similar properties and respecting the additional condition. More precisely, $a$ and $b$ satisfying $a<b$ of (\ref{thm:3.3.3.2}) correspond to the $a$-th and $b$-th ($\frac{n-1}{2}$)-dimensional PL spheres, respectively.

(\ref{thm:3.4}) is also one of new ingredients in this theorem. In any case, the disjoint union of all of the holes are obtained by the following as in the proof of (\ref{thm:3.3}). We first smoothly embed the disjoint union of the PL spheres whose dimensions are greater than $[\frac{n-1}{2}]$. After that we smoothly embed the disjoint union of the PL spheres whose dimensions are smaller than $[\frac{n-1}{2}]$. 

We also refer to the proof of (\ref{thm:3.3}) in completing our proof of (\ref{thm:3.4}).
For the homology groups and cohomology groups of the manifold which is very essentially an SIE-$(K,{\mathcal{M}}_{\rm \mathcal{CS}Diff})$ and {\rm (}$k-1${\rm )}-connected and whose holes are always homotopy spheres, we have similar (co)homology classes and properties taking suitable coefficients as the coefficients at the classes represented by the fibers of the bundles or the closed tubular neighborhoods. This completes the proof of (\ref{thm:3.4}).  

We can show (\ref{thm:3.5}) immediately by the assumptions that $n<3k$ and that the manifold is ($k-1$)-connected together with the definition and some fundamental properties of triple Massey products.
\end{proof}
Last, we present two more explicit cases as theorems. Hereafter, connected sums of smooth manifolds are considered in the smooth category.
\begin{Thm}
\label{thm:4}
In the situation of Theorem \ref{thm:2}, let $(n,k)=(5,2)$. If a root contains a diffeomorphism type for a manifold which is not diffeomorphic to any homotopy sphere, then it is for a manifold represented as a connected sum of finitely many copies of $S^2 \times S^2$. 

Moreover. let $(n,k)=(6,2)$ and consider only ''SEE''.  If the root contains a diffeomorphism type for a manifold which is not diffeomorphic to any homotopy sphere, then it is for a manifold represented as a connected sum of finitely many copies of $S^2 \times S^2$ or a $5$-dimensional closed, simply-connected and spin manifold in \cite{barden} {\rm (}\cite{smale}{\rm )}.
\end{Thm}
\begin{Thm}
\label{thm:5}
In the situation of Theorem \ref{thm:2}, let $(n,k)=(6,2)$. If a root contains a diffeomorphism type for a manifold which is not diffeomorphic to any homotopy sphere, then it is for a $4$-dimensional closed and simply-connected manifold whose signature is $0$ or a $5$-dimensional closed, simply-connected and spin manifold in \cite{barden} {\rm (}\cite{smale}{\rm )}.

\end{Thm}
\begin{proof}[Important ingredients for proofs of Theorems \ref{thm:4} and \ref{thm:5}.]
We only introduce arguments on embeddings and immersions of closed, simply-connected manifolds into Euclidean spaces.
If a closed and simply-connected manifold can be smoothly embedded into ${\mathbb{R}}^5$ and it is not a homotopy sphere, then it is $4$-dimensional and represented as a connected sum of finitely many copies of $S^2 \times S^2$ by virtue of \cite{freedman} and \cite{milnor} for example, together with the classical theory of characteristic classes \cite{milnorstasheff}, and that of embeddings for example. If a closed and simply-connected manifold can be smoothly embedded into ${\mathbb{R}}^6$ and it is not a homotopy sphere, then it is $4$-dimensional and represented as a connected sum of finitely many copies of $S^2 \times S^2$ as before or a $5$-dimensional closed, simply-connected and spin manifold by virtue of \cite{barden} and \cite{smale} (see also \cite{nishioka}).  If a closed and simply-connected manifold can be smoothly immersed into ${\mathbb{R}}^6$ and it is not diffeomorphic to a homotopy sphere, then it is a $4$-dimensional closed and simply-connected manifold whose signature is $0$ by virtue of similar theory or a $5$-dimensional closed, simply-connected and spin manifold as before.  
\end{proof}
 
Note that in these two theorems, to obtain more general immersed or embedded manifolds and special generic maps via Proposition \ref{prop:1} (\ref{prop:1.2}), we can consider more general holes in the homology classes represented by the holes than ones in Theorem \ref{thm:3}. 
According to \cite{kitazawa0.1}--\cite{kitazawa0.2}, \cite{kitazawa}, \cite{kitazawa2}--\cite{kitazawa7}, \cite{saeki} and \cite{saekisuzuoka} for example, if the absolute value of the codimension is sufficiently large, then for the manifolds admitting special generic maps, we can know invariants such as homology groups and cohomology rings well from these invariants of the immersed manifolds in the manifolds of the targets. More precisely, see Proposition 4 of \cite{kitazawa0.2} and Propositions 1 and 5 of \cite{kitazawa7} for example. Thus we can also obtain information of the manifolds admitting the special generic maps. 

As another explicit remark, for example, we can obtain various examples accounting for Theorem \ref{thm:1} well: see the original paper \cite{kitazawa9}.
\begin{Rem}
\label{rem:3}
Via the theory first presented in section 6 of \cite{saekisuzuoka} and the theory of the present section, we will have explicit smooth maps or so-called {\it fold} maps presented in \cite{kitazawa} and \cite{kitazawa2}--\cite{kitazawa6} for example. Rigorous understandings and proofs are left to readers.
\end{Rem}
\begin{Rem}
\label{rem:4}
We can see and we will be able to see more that with a little effort, we obtain special generic maps on some explicit manifolds with additional information on the topologies and the differentiable structures. However, for example, in the following two explicit cases, we do not know the answer to the following question: for closed and simply-connected manifolds, if we drop or weaken the condition that the special generic map $f$ and the smoothly immersed or embedded manifold $W_f$ are (essentially or very essentially) an SIE-$(K,{\mathcal{M}}_{\rm \mathcal{CS}Diff})$ or SEE-$(K,{\mathcal{M}}_{\rm \mathcal{CS}Diff})$, does the class of the manifolds admitting corresponding special generic maps become wider?

\begin{enumerate}
\item The dimension of the manifold of the domain is greater than $5$ and that of the target is $4$.  
\item The dimension of the manifold of the target is greater than $4$.  
\end{enumerate}

We do not know whether general studies on $4$-dimensional polyhedra or $5$-dimensional compact manifolds with non-empty boundaries such as \cite{whitehead} and studies of this type can help us to study these problems. We present a related explicit conjecture or a description of \cite{kitazawa8}. 
Remark 1 of \cite{kitazawa8} implicitly conjectures that $7$-dimensional closed, simply-connected and spin manifolds whose cohomology rings are isomorphic to that of a product of (a copy of) $S^3$ and (that of) the complex projective plane admit no special generic maps into ${\mathbb{R}}^n$ for $n=5,6$. These manifolds are also studied in \cite{wang}. We cannot deduce that such manifolds admit no special generic maps from Theorem \ref{thm:1} of the present paper.
\end{Rem}
\section{Acknowledgement.}
The author is a member of and supported by JSPS KAKENHI Grant Number JP17H06128 "Innovative research of geometric topology and singularities of differentiable mappings"(Principal investigator: Osamu Saeki). We declare that all data essentially related to our present study are all in the present paper.

\end{document}